\documentclass[11pt]{article}
\usepackage[a4paper]{geometry}
\usepackage{amsfonts}
\usepackage{amsmath}
\usepackage{amssymb}
\usepackage{amsthm}
\usepackage{fourier}
\usepackage{faktor}
\usepackage{graphicx}
\usepackage{physics}
\usepackage{enumitem}
\usepackage{pifont}
\usepackage{bbold}
\usepackage[font=small,labelfont=bf]{caption}
\usepackage{titlepic}
\usepackage{hyperref}
\hypersetup{linktoc=page}
\usepackage[capitalize, noabbrev]{cleveref}

\setitemize[0]{label=---}

\DeclareMathOperator{\Span}{Span}
\DeclareMathOperator{\Lie}{Lie}
\DeclareMathOperator{\Ad}{Ad}
\DeclareMathOperator{\ad}{ad}

\DeclareMathOperator{\PWC}{PWC}

\newcommand{\e}{\varepsilon}
\newcommand{\A}{\mathcal A}
\newcommand{\C}{{\mathbb C}}
\newcommand{\R}{{\mathbb R}}
\newcommand{\N}{{\mathbb N}}
\newcommand{\Z}{{\mathbb Z}}
\newcommand{\g}{{\mathfrak g}}

\newcommand{\li}{{\mathfrak l}}

\newcommand{\sll}{{\mathfrak {sl}}}
\newtheorem{theorem}{Theorem}
\newtheorem{lemma}{Lemma}
\newtheorem{prop}{Proposition}

\newtheorem{defi}{Definition}

\title{Good Lie Brackets for classical and quantum harmonic oscillators}
\author{Andrei Agrachev\thanks{SISSA, via Bonomea 265, 34136 Trieste, Italy (agrachev@sissa.it)} \and Bettina Kazandjian\thanks{Sorbonne Université, Université Paris Cité, CNRS, Inria, Laboratoire Jacques-Louis Lions, Paris, France (bettina.kazandjian@sorbonne-universite.fr)} \and Eugenio Pozzoli\thanks{Univ Rennes, CNRS, IRMAR - UMR 6625, F-35000 Rennes, France (eugenio.pozzoli@univ-rennes.fr)}}

\begin{document}

\maketitle

\begin{abstract}
We study the small-time controllability problem on the Lie groups $SL_2(\R)$ and $SL_2(\R)\ltimes H_{d}(\R)$ with Lie bracket methods (here $H_{d}(\R)$ denotes the $(2d+1)$-dimensional real Heisenberg group). Then, using unitary representations of $SL_2(\R)\ltimes H_{d}(\R)$ on $L^2(\R^d,\C)$ and $L^p(T^*\R^d,\R), p\in[1,\infty)$, we recover small-time approximate reachability properties of the Schrödinger PDE for the quantum harmonic oscillator, and find new small-time approximate reachability properties of the Liouville PDE for the classical harmonic oscillator. 
   
\end{abstract}


\section{Introduction}
\subsection{The model}
Let $G$ be a connected Lie group with Lie algebra $\g$. In this paper, we study left-invariant control affine systems of the form
\begin{equation} \label{eq1}
    \dot{q}(t)=q(t)\left(a+\sum_{i=1}^k u_i(t) b_i\right)\in T_{q(t)}G, \quad q\in G, a ,b_i \in \g , t \in [0,T].
\end{equation}
The controls are real-valued and piece-wise constant, $u=(u_1,...,u_k) \in \PWC([0,T],\R^{k})$. The solution of (\ref{eq1}) associated with the initial condition $q_0$ and the control $u$ at time $t$ is denoted by $q(q_0,u,t)$. When the initial condition $q_0={\rm id}_G$ is the identity of $G$, we drop it from the notation and simply write the associated solution as $q(u,t)$. 
\begin{defi} \label{def1}
\begin{itemize}
    \item An element $q \in G$ is \textbf{reachable} by system (\ref{eq1}) if there exist $T >0$ and $u \in \PWC([0,T],\R^{k})$ such that $q(u,T)=q$. The set of reachable elements is called the attainable set and it is denoted by $\A$.
    \item An element $q \in G$ is \textbf{approximately reachable} by system (\ref{eq1}) if $q \in \overline{\A}$.
    \item An element $q \in G$ is \textbf{small-time reachable} by system (\ref{eq1}) if for every $T>0$, there exists $\tau \in [0,T]$ and $u \in \PWC([0,\tau],\R^{k})$ such that $q(u,\tau)=q$. The set of small time reachable elements is denoted by $\A_{st}$.
    \item An element $q \in G$ is \textbf{small-time approximately reachable} by system (\ref{eq1}) if $q \in \overline{\A_{st}}$.
\end{itemize}
\end{defi}
\noindent The system is said to be \emph{controllable} (respectively \emph{approximately controllable}) if $\A=G$ (resp. if $\overline{\A}=G$). The system is said to be \emph{small-time controllable} (respectively \emph{small-time approximately controllable}) if $\A_{st}=G$ (resp. if $\overline{\A_{st}}=G$).

Note that the set $\A$ presented in \cref{def1} is the attainable set from the identity. Since system (\ref{eq1}) is left invariant, it is controllable from the identity if and only if it is controllable from any point of the group. The same statement holds true for small-time controllability.
\subsection{Good Lie brackets}
In this work we use the following notions of geometric control theory.
\begin{defi}
    An element $X \in \g$ is said to be a \emph{good Lie Bracket} for system (\ref{eq1}) if $e^{vX}\in \overline{\A_{st}}$ for every $v \in \R$. 
\end{defi}

\begin{defi}
    An element $X \in \g$ is said to be \emph{compatible} with system (\ref{eq1}) (at time $t$) if $e^{X}$ is approximately reachable by system (\ref{eq1}) (at time $t$). Another control system is said to be compatible with system (\ref{eq1}) if the closure of its attainable set is contained in $\overline{\A}$.
\end{defi}

It is well-known that every element in the (finite-dimensional) Lie algebra $\li:=\Lie \: \{b_1,...,b_k\}$, generated by $b_1,...,b_k$, is a good Lie bracket (see e.g. \cite[Lemma 6.2]{lierank} and \cite[Theorem 3.3]{small-time-domenico}).


It is also well-known that, if the drift of (\ref{eq1}) is recurrent or Poisson stable, then $-v a, v>0,$ is compatible with the system for a time large enough (see e.g. \cite[Proposition 8.2]{AS}). In this work, we shall study systems where the drift $a$ is a good Lie Bracket, a stronger property than compatibility, since the system is able to approximate the movement of the drift not only with positive and negative coefficients, but also in arbitrarily small times. The notion of good Lie bracket was recently introduced by the first author for more general control affine systems \cite{good-lie-brackets}, and studied also in the context of bilinear Schrödinger PDEs by Karine Beauchard and the third author \cite{beauchard-Pozzoli,beauchard-Pozzoli2}.

\subsection{A small-time controllability result on $SL_2\ltimes H_{d}$}
We introduce the following Lie algebras: 
\begin{itemize}
    \item $\frak{sl}_2(\R):=\left\{M\in {\rm Mat}_2(\R) \mid {\rm tr}(M)=0\right\}$;
        \item $\frak{h}_{d}(\R):=\left\{\begin{pmatrix}
        0&\vec{x}&z\\
        0&0_{d}&\vec{y}^\dagger\\
        0&0&0
    \end{pmatrix}\in {\rm Mat}_{d+2}(\R)\mid\vec{x},\vec{y}\in \R^d, z \in \R \right\}$.
\end{itemize}
Here, $0_d$ denotes the $d$-dimensional square zero matrix, $\vec{x}$ is a row $d$-dimensional vector and $\vec{y}^\dagger$ is a column $d$-dimensional vector. To ease notations, we shall simply denote them as $\frak{h}_{d}$ and ${\frak sl}_2$. Let $\frak{der}(\frak{h}_{d})$ denote the derivations over the algebra $\frak{h}_{d}$. We define a homomorphism $\rho:\frak{sl}_2\to \frak{der}(\frak{h}_{d})$ as 
$$\rho\begin{pmatrix}
        \alpha&\beta\\
        \gamma&-\alpha\\
    \end{pmatrix}\begin{pmatrix}
        0&\vec{x}&z\\
        0&0_{d}&\vec{y}^\dagger\\
        0&0&0
    \end{pmatrix}=\begin{pmatrix}
        0&\alpha\vec{x}+\beta \vec{y}&0\\
        0&0_{d}&\gamma \vec{x}^\dagger-\alpha\vec{y}^\dagger\\
        0&0&0
    \end{pmatrix}.$$
It induces a semi-direct product structure on $\frak{sl}_2\oplus \frak{h}_{d}$, called $\frak{sl}_2\ltimes_{\rho} \frak{h}_{d}$, with bracket
 \begin{equation}\label{eq:semi-direct-product}
 [(a,X),(b,Y)]=([a,b],[X,Y]+\rho(a)Y-\rho(b)X),\quad \forall a,b\in {\frak sl}_2, X,Y\in {\frak h}_d.
 \end{equation}
  We shall simply denote this Lie algebra by $\frak{sl}_2\ltimes \frak{h}_{d}$, and its associated Lie group by $SL_2\ltimes H_{d}$. We also fix basis $\{a,b,c\}$ of $\frak{sl}_2$ and $\{X_1,Y_1,\dots,X_d,Y_d,Z\}$ of $\frak{h}_{d}$ in such a way that the following commutation relations are satisfied $$[b,a]=c,[a,c]=2a,[b,c]=-2b,$$ 
$$[X_i,Y_i]=Z,[X_i,Z]=[Y_i,Z]=0, \quad i=1,\dots,d.$$ 
We first study the controllability of the following left-invariant control-affine system:
\begin{equation}\label{eq:SL}
\frac{d}{dt}q(t)=q(t)\left(a+u_0(t)b+\sum_{i=1}^d u_i(t)X_i+r(t)Z\right),\quad q\in SL_2\ltimes H_{d}.
\end{equation}
Our first result is the following small-time controllability property.
\begin{theorem}\label{thm:SL}
Equation \eqref{eq:SL} is small-time controllable on $SL_2\ltimes H_{d}$. Moreover, equation \eqref{eq:SL} with $u_i,r\equiv 0$ for $i\in\{1,\dots,d\}$ is small-time controllable on $SL_2$.
\end{theorem}
We remark that the second part of the theorem states the small-time controllability of a scalar-input system with drift on a non-compact connected Lie group $SL_2$: such a result would be impossible on compact Lie groups such as $SU(N), N\geq 2$, where scalar-input systems with drift are never small-time controllable (see, e.g., \cite{chambrion-agrachev,time-zero} and we refer also to \cite{gauthier-rossi} for recent advances of the subject).
\subsection{Consequences for harmonic oscillator Schrödinger equation}

The choice of studying $SL_2\ltimes H_{d}$ in this article is motivated by the description of two physical systems, namely the quantum and the classical harmonic oscillators. In this section we state the consequences for quantum harmonic oscillators. The relation between $SL_2$ and the Schrödinger equation for the harmonic oscillator is known as the Weil representation (see e.g. \cite[Chapter XI]{Lang}). \\

The Schrödinger equation for the quantum harmonic oscillator is given by
\begin{equation}\label{eq:schro}
i\partial_t\psi(t,x)=\left(-\frac{1}{2}\Delta + u_0(t)\frac{|x|^2}{2}+\sum_{j=1}^du_j(t)x_j\right)\psi(t,x),  \qquad \psi (t=0)=\psi_0 \in L^2(\R^d,\C),\end{equation}
with control $u=(u_0,\dots,u_d) \in \PWC([0,T],\R^{d+1})$. The control $u_0$ models the frequency tuning of the trapping potential, and the $u_j,j=1,\dots,d$ model the dipolar interaction with the spacial positions of the oscillator. Such system is among the most relevant quantum dynamics, widely used in physics and chemistry to model a variety of situations, such as atoms trapped in optical cavities or vibrational dynamics of molecular bonds.

The Lie algebra generated by the linear operators $\Lie \: \{\frac{i\Delta}{2},\frac{i|x|^2}{2}\}$ equipped with the commutator is isomorphic to $\sll_2(\R)$, the algebra of real $2\time 2$ traceless matrices, and the Lie algebra generated by $\Lie \: \{\frac{i\Delta}{2},\frac{i|x|^2}{2},ix_1,\dots,ix_d\}$ is isomorphic to ${\frak sl}_2 \ltimes{\frak h}_{d}$. We thus have introduced infinite-dimensional representations of the algebras ${\frak sl}_2,{\frak sl}_2 \ltimes{\frak h}_{d}$ acting on $L^2(\R^d,\C)$.

Quantum states are defined up to global phases, hence we shall say that a state $\psi_1$ is reachable if there exists a number $\theta\in[0,2\pi)$ such that $e^{i\theta}\psi_1$ is reachable. Since the generator $-\frac{1}{2}\Delta + u_0\frac{|x|^2}{2}+\sum_{j=1}^du_jx_j$ is essentially self-adjoint on $C^\infty_c(\R^d,\C)$ for any $u_0,\dots,u_d\in \R^{d+1}$ (see, e.g., \cite[Corollary page 199]{rs2}), equation \eqref{eq:schro} is globally-in-time well-posed, meaning that for any $\psi_0\in L^2(\R^d,\C),t\geq 0, u\in PWC([0,t],\R^{d+1})$ there exists a unique mild solution $\psi(t,u,\psi_0)\in C(\R,L^2(\R^d,\C))$ of \eqref{eq:schro}.
\begin{defi}
A state $\psi_1\in L^2(\R^d,\C)$ is said to be small-time approximately reachable from $\psi_0\in L^2(\R^d,\C)$ for system \eqref{eq:schro} if for every $\varepsilon>0$ there exist a time $\tau\in [0,\varepsilon]$, a control $u\in PWC([0,\tau],\R^{d+1})$ and a global phase $\theta\in[0,2\pi)$ such that the solution $\psi$ of \eqref{eq:schro} satisfies $\|\psi(\tau,u,\psi_0)-e^{i\theta}\psi_1\|_{L^2}<\varepsilon$.
\end{defi}

We have the following consequence of Theorem \ref{thm:SL}.
\begin{theorem}\label{thm:schro}
Each state of the family $\{e^{is\Delta}e^{i(\alpha |x|^2+p x)}\sigma^{d/2}\psi_0(\sigma x+\beta),s,\alpha\in\R,p,\beta\in \R^d,\sigma>0\}$ is small-time approximately reachable from $\psi_0\in L^2(\R^d,\C)$ for system \eqref{eq:schro}.
\end{theorem}
Such a result was already found in \cite{beauchard-Pozzoli} with analytical methods, and we propose here a more geometric proof, based on the unitary representations of $SL_2\ltimes H_{d}$. The statement of this theorem tells that we can control in small-times some physically relevant quantities, such as position, momentum, and spread of the initial state. The corresponding controllability analysis for system \eqref{eq:schro} with $u_0\equiv 1$ was performed in \cite{mirrahimi-rouchon}, and a weaker statement was obtained in \cite{teismann}. More general small-time controllability properties of Schrödinger PDEs were recently obtained in \cite{beauchard-Pozzoli2}.

\subsection{Consequences for harmonic oscillator Liouville equation}
In this section we state the consequences for classical harmonic oscillators. The Liouville equations for classical harmonic oscillator is given by the transport of a density $\rho\in L^p(T^*\R^d)$ along an Hamiltonian vector field
\begin{equation}\label{eq:liouville}
\partial_t \rho(t,q,p)=\vec{H}_u\rho(t,q,p), \quad \rho(t=0)=\rho_0\in L^p(T^*\R^d)
\end{equation}
where the Hamiltonian function is given by
\begin{equation}\label{eq:hamiltonian}
H_u(q,p)=\frac{|p|^2}{2}+u_0(t)\frac{|q|^2}{2}+\sum_{j=1}^du_j(t)q_j, 
\end{equation}
and the Hamiltonian vector field is the first-order differential operator defined as 
\begin{equation}\label{eq:hamiltonian-field}
\vec{H}_u:=\{H_u,\cdot\}=p\cdot\nabla_q-u_0(t)q\cdot\nabla_p-\sum_{j=1}^mu_j(t)\partial_{p_j},
\end{equation}
and $\{\cdot,\cdot\}$ denotes the Poisson bracket on $T^*\R^d$ (that is, $\{f,g\}=\partial_pf\partial_qg-\partial_q f\partial_pg$ for any $f=f(q,p),g=g(q,p)\in C^\infty(T^*\R^d)$). The Lie algebra of smooth functions generated by $\Lie \: \{\frac{|p|^2}{2},\frac{|q|^2}{2},q_1,\dots,q_d\}$, equipped with the Poisson bracket, is also isomorphic to $\sll_2\ltimes {\frak h_{d}}$. Since the vector field \eqref{eq:hamiltonian-field} is globally Lipschitz, the solution of the Liouville equation $\rho(t)$ is globally-in-time well-posed, meaning that for any $\rho_0\in L^p(T^*\R^d,\R),t\geq 0,u\in PWC([0,t],\R^{d+1})$ there exists a unique mild solution $\rho(t,u,\rho_0)\in C(\R,L^p(T^*\R^d,\R))$ of \eqref{eq:liouville}. Such a solution writes as $\rho(t,u,\rho_0)=\rho_0\circ \phi_{H_{u}}^t(q,p)$ where $\phi_{H_{u}}^t$ is the flow solving the Hamiltonian equations on $T^*\R^d$
\begin{equation}\label{eq:flow}
\frac{d\phi_{{H}_u}^t(q,p)}{dt}=\begin{pmatrix} \partial_p H_u(\phi_{{H}_u}^t(q,p))\\ -\partial_q H_u(\phi_{{H}_u}^t(q,p))  \end{pmatrix},\quad \phi_{{H}_u}^{t=0}(q,p)=(q,p).
\end{equation}
Notice that $\phi_{H_{u}}^t$ is an Hamiltonian diffeomorphism hence orientation- and volume-preserving.
\begin{defi}
A state $\rho_1\in L^p(T^*\R^d), p\in [1,\infty)$, is said to be small-time approximately reachable from $\rho_0\in L^p(T^*\R^d)$ for system \eqref{eq:liouville} if for every $\varepsilon>0$ there exist a time $\tau\in [0,\varepsilon]$ and a control $u\in PWC([0,\tau],\R^{d+1})$ such that the solution $\rho$ of \eqref{eq:liouville} satisfies $\|\rho(\tau,u,\rho_0)-\rho_1\|_{L^p}<\varepsilon$.
\end{defi}
We have the following consequence of Theorem \ref{thm:SL}.
\begin{theorem}\label{thm:liouville}
Each state of the family $\{\rho_0(\alpha (q+tp+s),\alpha^{-1}(p+rq+w))\mid \alpha>0,t,r\in\R, s,w\in\R^d\}$ is small-time approximately reachable from $\rho_0\in L^p(T^*\R^d)$ for system \eqref{eq:liouville}.
\end{theorem}

This result give some first insights on the controllability of Liouville Hamiltonian equations. They are obtained thanks to the specific nature of the Hamiltonian function \eqref{eq:hamiltonian}. Extensions of such technique to more general Hamiltonians will be the subject of future investigations.

\medskip

The article is organised as follows: in Section \ref{sec:compatible} we recall a useful compatibility property, and in Sections \ref{sec:thm1}, \ref{sec:thm2}, \ref{sec:thm3} we prove resp. Theorems \ref{thm:SL}, \ref{thm:schro}, and \ref{thm:liouville}.
\section{Compatible elements on Lie groups}\label{sec:compatible}

 In this section we recall a property (found by the first author in \cite[Corollary 1]{good-lie-brackets} for general control affine systems) in the specific Lie group framework introduced above.

We associate to the Lie algebra $\li:=\Lie \: \{b_1,...,b_k\}$ its Lie group
$$L:=\{e^{t_1X_1}...e^{t_mX_m} \: | \: m \in \N^{*},X_1,...,X_m \in \li,t_1,...,t_m\in\R\}.$$

 The following result is a particular case of \cite[Corollary 1]{good-lie-brackets}.

\begin{theorem} \label{th1}
    For $\tau \geq 0$, every convex combination of elements the form $Ad_{L}(\tau a) + \li$ is compatible with system (\ref{eq1}) at time $\tau$.
\end{theorem}
\begin{proof}
The driftless system 
\begin{equation} \label{eq2}
    \dot{q}(t)=q(t)\sum_{i=1}^k u_i(t) b_i \qquad q\in G, b_i \in \g, u_i(t) \in \R,
\end{equation}
is compatible with system (\ref{eq1}). More precisely, every $q \in L$ is small-time approximately reachable by (\ref{eq1}): indeed, thanks to \cite[Lemma 6.2]{lierank}, $q=e^{s_Nb_{i_N}}\dots e^{s_1b_{i_1}}$ for some $s_1,\dots,s_N\in \R, N\in \N$, $\{i_1,\dots,i_N\}\subset \{1,\dots,k\}$. Consider then a piecewise constant control $u$ defined as $u_{i_1}(t)\equiv s_1/t_1$ for a time interval $[0,t_1]$, ..., $u_{i_N}(t)\equiv s_N/t_N$ for a time interval $[0,t_N]$ : the associated solution writes $q(t_1+\dots +t_N,u)=e^{t_Na+s_Nb_{i_N}}\dots e^{t_1+s_1b_{i_1}}$ hence for $t_1,\dots,t_N$ small enough $q(t_1+\dots +t_N,u)$ is close to $q$.

Moreover, $e^{\tau a}$ is reachable in time $T=\tau$, using a free evolution of the system (i.e. a control $u=0$ on a time interval of size $\tau$). Let $q\in L$, then $qe^{\tau a}q^{-1}$ is approximately reachable in time $T= \tau+\e$ for every $\e > 0$. Thanks to the properties of the exponential map on a Lie group,
\begin{equation*}
    q e^{\tau a}q^{-1}=e^{\Ad_{q}(\tau a)}.
\end{equation*}
Thus $\Ad_{q}(\tau a)$ is compatible with (\ref{eq1}). Finally, according to standard relaxation technique, every convex combination of compatible vector fields is also compatible (see e.g. \cite[Proposition 8.1]{AS}). 
\end{proof}
We recall the following corollary of Krener's theorem for Lie Bracket Generating systems (see e.g. \cite[Corollary 8.1]{AS}).
\begin{prop} \label{prop1}
    If a system defined on a finite dimensional manifold is approximately controllable and Lie Bracket Generating then it is controllable.
\end{prop}
\section{Proof of Theorem \ref{thm:SL}}\label{sec:thm1}
Owing to \eqref{eq:semi-direct-product}, we have
$$[a,X_i]=Y_i,[b,X_i]=0,[c,X_i]=X_i,[a,Y_i]=0,[b,Y_i]=X_i,[c,Y_i]=-Y_i, \quad i=1,\dots,d,$$ 
$$[a,Z]=[b,Z]=[c,Z]=0,$$ 
where we simply denote $a$ and $X$ the elements $(a,0)$ and $(0,X)$ for any $a\in \frak{sl}_2, X\in \frak{h}_d$.
\subsection{Proof of the second statement of Theorem \ref{thm:SL}}
Thanks to Iwasawa decomposition, every element of $SL_2$ can be written in the form $e^{t_1(a-b)}e^{t_2 c}e^{t_3b}$ where $t_1,t_2,t_3 \in \R$ (and this decomposition is unique).

 System (\ref{eq:SL}) can then be lifted to every covering space of $SL_2\ltimes H_{d}$ (notice that $H_{d}$ is simply connected). The Iwasawa decomposition gives the universal simply connected covering space: $\R^3 \rightarrow SL_2$, $(t_1,t_2,t_3)\mapsto e^{t_1(a-b)}e^{t_2 c}e^{t_3b}$. The covering space of degree $m\in \N^*$ is obtained by identification of elements $e^{t(a-b)}$ for every $t\in 2\pi m \Z$.
\begin{theorem} \label{th3}
    The elements $b$ and $c$ are good Lie brackets and for every $t_1\geq0$, $e^{t_1(a-b)}$ is small-time approximately reachable for system (\ref{eq:SL}) defined on any covering space of $SL_2$. System \eqref{eq:SL} with $u_i,r\equiv 0, i\in\{1,\dots,d\}$ is small-time controllable on every covering space of $SL_2$ except the universal one.  
\end{theorem}
\begin{proof}
We apply \cref{th1} with $\li=\Span \: \{b\}$ and $L=\{e^{vb} \: | \: v\in \R\}$. Every element of the form $\Ad_{e^{vb}}(\tau a)+ub$ with $u,v\in \R, \tau>0$ is compatible. Thanks to the properties of the exponential map on a Lie group,
\begin{align*}
    \Ad_{e^{vb}}(\tau a)&=e^{\ad_{vb}}(\tau a) 
    =\sum_{k=0}^{+\infty}\frac{v^k}{k!}(ad_b)^k(\tau a) 
    =\tau a + v\tau c - v^2\tau  b.
\end{align*}
By taking $v=r/\tau,$ the elements $\tau a+rc$ are compatible for every $r \in \R$, hence $c$ is a good Lie bracket.
Then, the following system
\begin{equation} \label{eq5}
    \dot{q}=q(a+ub+vc) \qquad u,v \in \R,
\end{equation}
is compatible with system (\ref{eq1}).  We apply \cref{th1} to the system (\ref{eq5}), with $\li=\Span \: \{b,c\}$. Then every element of the form $\Ad_{e^{ub+vc}}(a)+sb+tc$, $u,v,s,t\in \R$ is compatible. In particular, $\Ad_{e^{vc}}(a)=\exp(\ad_{vc}(a))=e^{-2v}a$ is compatible for every $v\in \R$. Thus $wa$ is also compatible for every $w>0$. Then every element of the form $e^{t_1(a-b)}e^{t_2c}e^{t_3b}$, $t_1\geq0$, $t_2,t_3 \in \R$ is small-time reachable by the compatible system, thus it is small-time approximately reachable by system (\ref{eq1}). If the considered cover is not the universal one, for every $t\in \R$ there exists $t'\geq0$ such that $e^{t(a-b)}=e^{t'(a-b)}$. Hence (\ref{eq1}) is small-time approximately controllable on every cover except the universal one. Moreover, system (\ref{eq:SL}) is Lie Bracket Generating so, according to \cref{prop1}, it is small-time controllable on every covering space of $SL_2$ except the universal one. 
\end{proof}
\subsection{Proof of the first statement of Theorem \ref{thm:SL}}
Thanks to \cite[Lemma 6.2]{lierank}, and since (\ref{eq:SL}) is Lie Bracket Generating, we are left to prove that $X_i,Y_i, i=1,\dots,d$ and $Z$ are good Lie brackets. It is clear that $X_i,i\in \{1,\dots,d\},Z$ are good Lie brackets. Thanks to Theorem \ref{th1}, every element of the form $\Ad_{e^{vX_i}}(a)+uZ$ with $u,v\in \R$ is compatible. Thanks to the properties of the exponential map on a Lie group,
\begin{align*}
    \Ad_{e^{vX_i}}(\tau a)&=e^{\ad_{vX_i}}(\tau a) 
    =\sum_{k=0}^{+\infty}\frac{v^k}{k!}(ad_{X_i})^k(\tau a) 
    =\tau a + v\tau Y_i - \frac{v^2\tau}{2}  Z.
\end{align*}
By taking $v=r/\tau, u=v^2\tau/2,$ the elements $\tau a+rY_i$ are compatible for every $r \in \R,\tau>0$, hence by taking $\tau$ small enough $Y_i$ is a good Lie bracket.

\section{Proof of Theorem \ref{thm:schro}}\label{sec:thm2}

A direct computation shows that 
$$\Lie \: \left\{\frac{i\Delta}{2},\frac{i|x|^2}{2}\right\}={\rm span} \:\left\{\frac{i\Delta}{2},\frac{i|x|^2}{2},\frac{d}{2}+x\cdot \nabla\right\}.$$ Every element of this Lie algebra is an essentially skew-adjoint operator on $L^2(\R^d,\C)$ with domain $\mathcal{C}_c^{\infty}(\R^n,\C).$ Moreover, the application $\pi : \sll_2(\R) \rightarrow \Lie \: \{\frac{i\Delta}{2},\frac{i|x|^2}{2}\}$ defined by $\pi(a)=\frac{i\Delta}{2}$, $\pi(b)=\frac{i|x|^2}{2}$, $\pi(c)=\frac{d}{2}+x\cdot \nabla$ and extended by linearity, is a Lie algebra isomorphism.

In Theorem \ref{th3} we showed the small-time controllability on every covering space of $SL_2$ (except the universal one). In the following section we translate this property for the Schrödinger equation \eqref{eq:schro}.  We denote with $\mathbb{U}(X)$ the group of linear unitary operators on a normed space $X$.
\subsection{Unitary representation of $SL_2(\R)$ and its covering spaces}

\begin{prop} \label{th6}
    There exists a group morphism $f$ from the universal cover $\widetilde{SL}_2$ to $\mathbb{U}(L^2(\R^d,\C))$ such that $f(e^{a})=e^{\pi(a)}$ for every $a\in \frak{sl}(2)$.
Moreover, $f$ is strongly continuous, i.e. for every $\psi \in L^2(\R^d,\C)$ the function $g \in \widetilde{SL}_2 \mapsto f(g)\psi$ is continuous.
\end{prop}
\begin{proof} The construction of an infinite-dimensional unitary representation $f$ satisfying $f(e^{a})=e^{\pi(a)}$ for every $a\in \frak{sl}(2)$ is standard. For the proof of the strong continuity, we use the following lemma (see, e.g., \cite[Theorem VIII.21 \& Theorem VIII.25(a)]{rs1}).
\begin{lemma}\label{lemm3}
    Let $(A_k)_{k\in \N}$, $A$ be self-adjoint operators on a Hilbert space $\mathcal{H}$, with a common core $D$. If $||(A_k-A)\psi||_{\mathcal{H}}\underset{k\rightarrow +\infty}{\longrightarrow}0$ for any $\psi \in D$, then $||(e^{iA_k}-e^{iA})\psi||_{\mathcal{H}}\underset{k\rightarrow +\infty}{\longrightarrow}0$ for any $\psi \in \mathcal{H}$.
\end{lemma}
Consider $(g_k)_{k\in \N}$ a sequence in $\widetilde{SL}_2$ such that $g_k \underset{k \rightarrow +\infty}{\longrightarrow} g$. According to the Iwasawa decomposition there exist three converging sequences in $\R$, $t^k_{j} \underset{k \rightarrow +\infty}{\longrightarrow}t_j$, $j \in \llbracket 1,3 \rrbracket$, such that $g_k=e^{t^k_{1}(a-b)}e^{t_{2}^k c}e^{t^k_{3} b}$ and $g=e^{t_1(a-b)}e^{t_2 c}e^{t_3 b}$. We have
\begin{align*}
    f(g_k)&=f(e^{t^k_{1}(a-b)}e^{t_{2}^k c}e^{t^k_{3} b}) 
    =f(e^{t^k_{1}(a-b)})f(e^{t_{2}^k c})f(e^{t^k_{3} b}) 
    =e^{\pi(t^k_{1}(a-b))}e^{\pi(t_{2}^k c)}e^{\pi(t^k_{3} b)} \\
    &=e^{it^k_{1}\frac{(\Delta - |x|^2)}{2}}e^{t_{2}^k(\frac{n}{2}+x\cdot \nabla)}e^{it_{3}^k\frac{|x|^2}{2}}.
\end{align*}
Clearly, for every $\varphi \in C^\infty_c(\R^d,\C)$, 
\begin{align*}
&\left\|t_1^k\left(\frac{\Delta-|x|^2}{2}\right)\varphi - t_1\left(\frac{\Delta-|x|^2}{2}\right)\varphi\right\|_{L^2},\left\|t_2^k\left(\frac{d}{2}+x\cdot \nabla\right)\varphi - t_2\left(\frac{d}{2}+x\cdot \nabla\right)\varphi\right\|_{L^2},\\
&\left\|t_3^k\left(\frac{|x|^2}{2}\right)\varphi - t_3\left(\frac{|x|^2}{2}\right)\varphi\right\|_{L^2} \underset{k\rightarrow + \infty}{\longrightarrow} 0, 
\end{align*}
Thus thanks to \cref{lemm3}, for every $\varphi \in L^2(\R^n,\C)$, $$||e^{it_{1}^k\frac{(\Delta - |x|^2)}{2}}\varphi - e^{it_{1}\frac{(\Delta - |x|^2)}{2}}\varphi||_{L^2},\|e^{it_{2}^k(\frac{d}{2}+x\cdot \nabla)}\varphi - e^{it_{2}\frac{d}{2}+x\cdot \nabla}\varphi||_{L^2},\|e^{it_{3}^k\frac{|x|^2}{2}}\varphi - e^{it_{3}\frac{|x|^2}{2}}\varphi||_{L^2} \underset{k\rightarrow + \infty}{\longrightarrow} 0.$$
Hence for every $\varphi \in L^2(\R^n,\C)$,
\begin{align*}
    f(g_k)\varphi&=e^{it^k_{1}\frac{(\Delta - |x|^2)}{2}}e^{t_{2}^k(\frac{n}{2}+x\cdot \nabla)}e^{it_{3}^k\frac{|x|^2}{2}}\varphi \\
    &\underset{k\rightarrow + \infty}{\longrightarrow} e^{it_{1}\frac{(\Delta - |x|^2)}{2}}e^{t_{2}(\frac{n}{2}+x\cdot \nabla)}e^{it_{3}\frac{|x|^2}{2}}\varphi=f(g)\varphi,
\end{align*}so $f$ is strongly continuous. $\blacksquare$ \end{proof}

\begin{lemma}\label{lem:metaplectic}
    If $d$ is even, there exists a unitary representation $f:\widetilde{SL}_2 \rightarrow \mathbb{U}(L^2(\R^d))$ such that $f(e^{a})=e^{\pi(a)}$ for every covering space $\widetilde{SL}_2$ of $SL_2$ (including $SL_2$ itself). If $d$ is odd, such a unitary representation $f$ is obtained if $\widetilde{SL}_2$ is a covering space of degree $2m, m\in \N^{*}$, or the universal cover. In particular, if $d$ is odd, we have a unitary representation from the two-sheets cover of $SL_2$, also know as the metaplectic representation.
\end{lemma}

\begin{proof} We denote $f:\widetilde{SL}_2 \rightarrow \mathbb{U}(L^2)$ the unitary representation of the universal cover, which exists according to \cref{th6} and which verifies $f(e^{a})=e^{\pi(a)}$. In order to obtain a representation for a covering space of degree $m\in \N^{*}$ of $SL_2$, we have to check that $f$ is well-defined on every homotopy class of the group. Every loop $\gamma \in SL_2$ is homotopic to one of the loops $[0,1]\ni t\mapsto w_k(t)=e^{2\pi kt(a-b)}, k\in\N$. If $L^2(\R^d,\C)\ni\varphi=\underset{{j\in\N^d}}{\sum}c_{j}\varphi_{j}$ is decomposed on the Hilbert basis of Hermite functions $\{\varphi_{j}\}_{j \in \N^d}$, then
    \begin{equation} \label{eq7}
        e^{\frac{it(\Delta-|x|^2)}{2}}\varphi=\sum_{(j_1,\dots,j_d)\in N^d}e^{-it(j_1+...+j_d+\frac{d}{2})}c_{j}\varphi_{j}.
    \end{equation}
According to (\ref{eq7}), we obtain $f(w_k(1))=e^{{i\pi k(\Delta-|x|^2)}}=(-1)^{kd}I$. If $d$ is even, then the definition of $f$ coincide on every $w_k, k\in \N^{*}$. If $d$ is odd, this is true for the loops of even degree, and so we obtain a representation of the covering space of degree $2m$, $m \in \N^*$. 
\end{proof}

\subsection{Conclusion of the proof of Theorem \ref{thm:schro}}
Let $\widetilde{SL_2}$ be the metaplectic group, for which there exists a unitary representation whatever the dimension $d$ of the space is (cf. Lemma \ref{lem:metaplectic}). Since $H_{d}$ is simply connected, the strongly continuous unitary representation $f$ can be extended in a standard way to $\widetilde{SL_2}\ltimes H_{d}$. It satisfies $f(e^{tX_i})=e^{it x_i}, f(e^{tY_i})=e^{t \partial_{x_i}},e^{tZ}=e^{it}, t\in \R$. 

Given an arbitrary state $\psi_1$ in the set of states described in Theorem \ref{thm:schro}, there exists an element $g$ in $\widetilde{SL}_2\ltimes H_{d}$ such that $\psi_1=f(g)\psi_0$. Since the element $g$ is small-time approximately reachable (cf. Theorem \ref{th3}), for any $\delta>0$ there exists a control $u \in PWC([0,\tau],\R^{d+1}), \tau\in[0,\delta],$ such that $|q(u,\tau)-g|<\delta$. Since $f$ is strongly continuous, for any $\varepsilon>0$ there exists $\delta>0$ such that if $|q(u,\tau)-g|<\delta$ then $\|f(q(u,\tau))\psi_0-f(g)\psi_0\|_{L^2}<\varepsilon$. Since $f$ is a morphism and $u$ is piecewise constant, $f(q(u,\tau))\psi_0=\psi(u,\tau,\psi_0)$.  Hence $\|\psi(u,\tau,\psi_0)-\psi_1\|_{L^2}=\|f(q(u,\tau))\psi_0-f(g)\psi_0\|_{L^2}<\varepsilon$, and the proof of Theorem \ref{thm:schro} is concluded.

\section{Proof of Theorem \ref{thm:liouville}}\label{sec:thm3}
For proving Theorem \ref{thm:liouville}, one follows the same exact lines of the proof of Theorem \ref{thm:schro}; we only point out the needed modifications. First of all a direct computation shows that
$${\rm Lie}\{|p|^2/2,|q|^2/2,q_1,\dots,q_d\}={\rm span}\{|p|^2/2,|q|^2/2,q_1,\dots,q_d,p\cdot q,p_1,\dots,p_d,1\}.$$
We thus have an isomorphism of Lie algebras $\pi:\frak{sl}_2\ltimes\frak{h}_{d}\to {\rm Lie}\{|p|^2/2,|q|^2/2,q_1,\dots,q_d\}$, and a group morphism $f:\widetilde{SL_2}\ltimes H_{d}\to \mathbb{U}(L^p(T^* \R^d))$ satisfying $f(e^{a})=\phi^1_{\pi(a)}$ for any $a\in \frak{sl}_2\ltimes\frak{h}_{d}$, and $\phi^t_{H(p,q)}$ acts on $\rho_0\in L^p(T^*\R^d)$ as $\rho_0\circ \phi^t_{H(p,q)}$. Such action is strongly continuous thanks to the dominated convergence theorem. Finally, as a slight difference w.r.t. to Lemma \ref{lem:metaplectic}, notice that $\phi^{t}_{(p^2-q^2)/2}(q_0,p_0)=(q_0\cos(t)+p_0\sin(t),-q_0\sin(t)+p_0\cos(t))$ hence $\phi^{2\pi k}_{(p^2-q^2)/2}=I$, irrespectively of $d$ being odd or even; the corresponding Lemma \ref{lem:metaplectic} for the Liouville representation $f:\widetilde{SL_2}\ltimes H_{d}\to \mathbb{U}(L^p(T^* \R^d))$ hence holds for covering spaces of even and odd dimensions.

\medskip

\textbf{Acknowledgments.} E.P. thanks the SMAI for supporting and the CIRM for hosting the BOUM project ”Small-time controllability of Liouville transport equations along an Hamiltonian field”, where some ideas of this work were conceived.
This research has been funded in whole or in part by the French National Research Agency (ANR) as part of the QuBiCCS project ”ANR-24-CE40-3008-01”.
This project has received financial support from the CNRS through the MITI interdisciplinary programs.

\bibliographystyle{siamplain}
\bibliography{references}
\end{document}